\documentclass{article}
\usepackage[utf8]{inputenc}
\usepackage[T1]{fontenc}
\usepackage{lmodern}
\usepackage{amsmath,amssymb,amsthm,mathtools}
\usepackage[margin=1in]{geometry}
\usepackage{microtype}
\usepackage{enumitem}
\usepackage[hidelinks]{hyperref}
\usepackage{comment}

\newtheorem{theorem}{Theorem}
\newtheorem{proposition}{Proposition}
\newtheorem{corollary}{Corollary}
\newtheorem{lemma}{Lemma}
\newtheorem{remark}{Remark}
\theoremstyle{definition}
\newtheorem{definition}{Definition}

\newcommand{\Z}{\mathbb Z}
\newcommand{\Q}{\mathbb Q}
\newcommand{\supp}{\operatorname{supp}}
\newcommand{\Dio}{\mathrm D}

\newcommand{\Cent}{\operatorname{C}}

\title{The Diophantine Problem in Wreath Products of Groups}
\author{Michael Catli\footnote{Stevens Institute of Technology},
Alexei Myasnikov\footnote{Stevens Institute of Technology}}

\begin{document}
\maketitle
\tableofcontents

\section{Introduction}
\label{se:introduction}

We study the Diophantine problem in restricted wreath products of groups.
The main results, summarized in Theorem \ref{th:general}, show that in a wide class of wreath products of groups, the Diophantine problem is undecidable.

Throughout, we write $G=B\wr H$, where $B$ is the passive factor, and $H$
is the active factor. 

We use two methods. The first is the retract principle: the decidability of
systems of equations with coefficients in a retract of a group is unchanged when solutions are
allowed in the ambient group.  We show that if both factors are
indicable, their wreath product has a retract isomorphic to
$\Z\wr\Z$. The Diophantine problem in this latter group is undecidable
by Myasnikov and Romanovskii~\cite{MR}; see also the explanation of
this result in~\cite[Section~1]{KM2025}. This gives undecidability for
wreath products of finitely generated indicable groups.

The second method combines e-definability with retractions.
Kharlampovich and Miasnikov~\cite{KM2025} proved undecidability for
right-iterated wreath products of nontrivial free abelian groups of
finite rank. Here we generalize this result considerably and prove that $\Dio(B\wr H)$ is undecidable whenever
$B$ is an infinite finitely generated abelian group and $H$ is a
finitely generated group containing an infinite e-definable finitely
generated abelian subgroup $A$. No torsion-free hypothesis on $B$,
$A$ or $H$ is required. If $\pi:B\wr H\to H$ is the canonical
retraction, then $\pi^{-1}(A)$ is e-definable and retracts onto
$B\wr A$. Since $B$ and $A$ are indicable, the first method gives
undecidability in $B\wr A$, which transfers to $B\wr H$.  A similar argument shows that for any groups $B$ and $H$, the Diophantine problems in $B$ and in $H$ are polynomial time reducible to the Diophantine problem in $B \wr H$.  In particular, if the Diophantine problem either in $B$ or in $H$ is undecidable, then the Diophantine problem in  $B \wr H$ is undecidable.

We also describe lower central quotients of wreath products with free
abelian factors and give a nilpotent quotient criterion for the undecidability of the Diophantine problem in such groups. Its undecidability comes from two results. Garreta, Miasnikov and
Ovchinnikov~\cite{GMO} showed that rings of integers of number fields
are e-interpretable in several classes of groups, including finitely
generated non-virtually abelian nilpotent groups. Koymans and
Pagano~\cite[Theorem~1.2]{KP} proved that $\Z$ is Diophantine in the
ring of integers of every number field. Consequently, these rings
have an undecidable Diophantine problem. Now, the transitivity of e-interpretability implies the undecidability of the Diophantine problem in finitely
generated non-virtually abelian nilpotent groups. The explicit lower central quotients give an
alternative proof of undecidability when both factors are nontrivial
free abelian groups of finite rank. We also obtain examples involving
centralizers, centers, and semidirect products.

In fact, there are not many examples of finitely generated infinite wreath products of two nontrivial
groups with a decidable Diophantine problem. The simplest examples arise
when the acting factor is finite. The result by Ciobanu, Holt, and Rees from ~\cite[Lemma~3.6]{CiobanuHoltRees2020} supplies a
general construction. 
Let $B$ be a finitely generated group with a decidable Diophantine problem,
and let $H$ be a finite group. Then the Diophantine problem in $W=B\wr H$
is decidable. 

We are not aware of any wreath products of finitely generated non-trivial groups, where the active group is infinite, with a decidable Diophantine problem.

It seems that the most interesting groups related to wreath products where the decidability of the Diophantine problem is unknown are the lamplighter group $\Z_2 \wr \Z$, the first Grigorchuk group $\Gamma$, and the wreath product $\Z_2 \wr \Gamma$. 

Although the results in this paper are only concerned with general systems of equations, we remark that in groups for which general systems of equations are undecidable, it is interesting to consider the decidability of special types of equations. For example, quadratic equations in the first Grigorchuk group $\Gamma$ are decidable \cite{LMU1}. We mention more results of this kind: it was proven in \cite{UW1} that one-variable equations in the Lamplighter group $L_2=\Z_2 \wr \Z$ are decidable (although the Diophantine problem remains open), and in the same group quadratic equations are decidable \cite{UW2}. Ushakov and Weiers also proved that solving orientable quadratic equations is decidable in wreath products of finitely generated abelian groups \cite{UW3} - despite the general Diophantine problem being undecidable \cite{KM2025}. Furthermore, the Diophantine problem for orientable quadratic equations in free metabelian groups was shown to be NP-complete \cite{LU1}. For results on the decidability of quadratic equations in metabelian Baumslag-Solitar groups, refer to \cite{MU2}. We also acknowledge several results concerning other classes of equations: Ushakov proved that solving spherical equations in the class of finite metabelian groups $G_{p, n}=\Z_p^n\rtimes \Z_p^*$ with $n\in \mathbb{N}$ and $p$ prime is computationally hard on average, assuming some lattice approximation problem is hard in the worst case \cite{U2}. For more results on spherical equations and their variants see \cite{LU2} and \cite{UMW1}. Solving systems of algebraic equations with exponents has also been proven to be NP-complete in \cite{MU1} where coefficients are elements of an algebraic closure of $\Q$ not equal to $0$ or $1$, and solutions are integers.

\section{Diophantine problems}
\label{se:Diophantine}

\subsection{The Diophantine problem}

The \emph{Diophantine problem} in a fixed algebraic structure
$\mathcal M$, denoted $\Dio(\mathcal M)$, asks whether there is an
algorithm that decides, for every finite system of equations with
coefficients in $\mathcal M$, whether that system has a solution in
$\mathcal M$. In a group, equations can be written as $w(X)=1$, where
$w$ is a word in the variables, their inverses, and the coefficients. 

Recall that for algebraic structures  $\mathcal A$ and $\mathcal M$  the Diophantine problem $\Dio(\mathcal A)$ m-reduces to the Diophantine problem $\Dio(\mathcal M)$, symbolically $\Dio(\mathcal A) \leq_m \Dio(\mathcal M)$, if there is a computable map $f$, which for every finite system $S$ over $\mathcal A$ 
computes a finite system $S^*$ over
$\mathcal M$ such that the system $S(\bar x)$ has a solution in  $\mathcal A$ if and only if  $S^*(\bar y)$ has a solution in $\mathcal M$. If the reduction $f$ is computable in polynomial time we write $\Dio(\mathcal A) \leq_P \Dio(\mathcal M)$.

For algorithmic and complexity statements, we use fixed finitely
generated structures in finite languages with fix finite generating
sets. Coefficients are specified by terms in these generators; for
groups, they are specified by words. Input length includes the
lengths of these terms or words.  An $m$ reduction is understood with respect
to this encoding. Purely algebraic statements about arbitrary
groups do not require finite generation; whenever a complexity
conclusion is drawn, the finite generation and encoding conventions
just stated apply. All equations in this paper may have coefficients.

\subsection{Interpretability by equations}

We recall e-interpretability and the facts about it needed below;
see~\cite{GMO}.

\begin{definition}
A subset $D\subseteq M^m$ is \emph{Diophantine}, or
\emph{e-definable}, in $\mathcal M$ if there is a finite system of
equations $\Sigma(\bar x,\bar y)$, with coefficients in $\mathcal M$,
such that
\[
D=\{\bar a\in M^m:\mathcal M\models
\exists\bar y\,\Sigma(\bar a,\bar y)\}.
\]
An existentially quantified conjunction of equations is called a
\emph{Diophantine formula}. These formulas are positive primitive
formulas; in purely functional languages, the two notions coincide.
\end{definition}

\begin{definition}
An algebraic structure $\mathcal A$ is \emph{e-interpretable} in
$\mathcal M$ if there are an e-definable set $D\subseteq M^n$ and a
surjection $\phi:D\to A$, called the \emph{coordinate map}, such
that the following sets are e-definable in $\mathcal M$:
\begin{enumerate}
\item the preimage of equality on $A$;
\item the preimage of the graph of every basic operation of
      $\mathcal A$;
\item the preimage of every basic relation of $\mathcal A$;
\item the preimage of every named constant of $\mathcal A$.
\end{enumerate}
Here preimages of relations and operation graphs are taken under
the appropriate coordinatewise powers of $\phi$.
\end{definition}

The composite of two e-interpretations is again an
e-interpretation: replace each equation in the outer interpreting
formulas by its interpretation in the inner structure and quantify
all the resulting witnesses. We will mostly use the following particular case of e-interpretability.

\begin{lemma}\label{eq:quotientInterp}
Let $N \leq K \leq G$ be  e-definable subgroups of a group $G$ and $N$ is normal in $K$. Then $G/N$ is e-interpretable in $G$.
\end{lemma}

\begin{proof}
Use the e-definable subset $K \subseteq G$ as the domain $D$ in the definition of e-interpretability above. Put  $\phi(x)=xN$ as the coordinate map $K \to K/N$. Then
$$
\phi(x) = \phi(y) \Longleftrightarrow x^{-1}y \in N,
$$
$$
xN \cdot yN = zN \Longleftrightarrow  z^{-1}xy \in N,
$$
$$
\phi(x) = N \Longleftrightarrow x \in N.
$$
Since $N$ is e-definable in $G$, this gives an e-interpretation of $K/N$ in $G$.
\end{proof}
The next result, established in \cite{GMO}, is central to our considerations.
\begin{lemma}\cite{GMO} \label{lem:interpretation-reduction}
An e-interpretation of $\mathcal A$ in $\mathcal M$ induces a
polynomial-time reduction
\[
\Dio(\mathcal A)\leq_P\Dio(\mathcal M).
\]
More precisely, a finite system $S(\bar x)$ over $\mathcal A$ can
be effectively transformed into a finite system $S^*(\bar y,\bar z)$ over
$\mathcal M$ such that the coordinate map $\phi: D \to \mathcal A$ sends every solution of $S^*(\bar y,\bar z)$  in $\mathcal M$ into a solution of $S(\bar x)$  in $\mathcal A$, and every solution of $S(\bar x)$  in $\mathcal A$ can be obtained this way. 
\end{lemma}

\begin{corollary}\label{karp}
If $\mathcal A$ is e-interpretable in $\mathcal M$ and
$\Dio(\mathcal A)$ is undecidable, then $\Dio(\mathcal M)$ is
undecidable.
\end{corollary}

\subsubsection{Finite verbal width}
\label{sse:Verbal}

For a group word $w=w(x_1,\ldots,x_m)$, let
\[
w(G)=\{w(g_1,\ldots,g_m):g_1,\ldots,g_m\in G\}.
\]
The \emph{verbal subgroup} associated with $w$ is
$\langle w(G)\rangle$. The word has \emph{finite width} in $G$
if every element of this subgroup is a product of at most $n$
elements of $w(G)^{\pm1}$ for some fixed $n$. In that case
$\langle w(G)\rangle$ is e-definable by
\begin{equation}\label{eq:verbal1}
\exists (y_{ij})\exists(z_{ij})\quad
x=\prod_{i=1}^{n}
\bigl(w(y_{i1},\ldots,y_{im})
w(z_{i1},\ldots,z_{im})^{-1}\bigr).
\end{equation}
Indeed, each positive or negative word value can occupy one factor
in this product by setting the other tuple equal to $1$, and unused
factors can also be made trivial. Conversely, every such product
belongs to the verbal subgroup.

A group is \emph{verbally elliptic} if every word has finite width
in it. Finitely generated nilpotent, metabelian and polycyclic
groups are verbally elliptic; see~\cite[Section~2.3]{GMO} and the
references there, including~\cite{Stroud} for the nilpotent case.
We use the lower central series convention
\[
\gamma_1(G)=G,\qquad
\gamma_{i+1}(G)=[\gamma_i(G),G].
\]
The subgroup $\gamma_i(G)$ is verbal for the left-normed
commutator word $[x_1,\ldots,x_i]$ (with $w_1(x_1)=x_1$).
Consequently all these subgroups are e-definable in the three
classes just mentioned, and their quotients are e-interpretable by
Lemma~\ref{eq:quotientInterp}.

\subsubsection{The nilpotent quotient criterion}
\label{se:Mainresults}

The following general result  will be handy in our arguments below.
\begin{theorem}\label{thm:nilpotent-criterion}
Let $G$ be finitely generated. Suppose that, for some $i\geq3$,
the subgroup $\gamma_i(G)$ is e-definable in $G$ and
$G/\gamma_i(G)$ is not virtually abelian. Then $\Dio(G)$ is
undecidable.
\end{theorem}

\begin{proof}
The quotient is a finitely generated nilpotent group and is
e-interpretable in $G$. By~\cite[Theorem~4.13]{GMO}, there is a
ring of integers $\mathcal O_K$ of a number field $K$ that is
e-interpretable in $G$. By~\cite[Theorem~1.2]{KP}, $\Z$ is
Diophantine in $\mathcal O_K$. Its ordinary addition and
multiplication are restrictions of the ring operations, so this
also gives an e-interpretation of the integer ring.

The e-interpretation of the integer ring $\Z$ in $\mathcal O_K$,
together with the e-interpretation of $\mathcal O_K$ in $G$, gives
the polynomial-time reductions
\[
\Dio(\Z; +,\cdot)
\leq_P
\Dio(\mathcal O_K)
\leq_P
\Dio(G).
\]
Indeed, each reduction follows from the reduction lemma for
e-interpretations (Lemma~\ref{lem:interpretation-reduction}).
Since Hilbert's tenth problem over $\Z$ is undecidable, it follows
that $\Dio(G)$ is undecidable.
\end{proof}

\begin{corollary}\label{t}
Let $G$ be a finitely generated metabelian group. If
$G/\gamma_i(G)$ is not virtually abelian for some $i\geq3$, then
$\Dio(G)$ is undecidable. The same conclusion holds for finitely
generated nilpotent or polycyclic groups satisfying this quotient
condition.
\end{corollary}

\begin{proof}
Finite verbal width makes $\gamma_i(G)$ e-definable, so
Theorem~\ref{thm:nilpotent-criterion} applies.
\end{proof}

\subsection{Retracts and the Diophantine problem}

\begin{definition}
A subgroup $H\leq G$ is a \emph{retract} of $G$ if there is a
homomorphism $\rho:G\to H$ whose restriction to $H$ is the
identity. Such a homomorphism is called a \emph{retraction}.
\end{definition}

A retraction gives $G=\ker\rho\rtimes H$: for $g\in G$,
$g=(g\rho(g)^{-1})\rho(g)$ and $\ker\rho\cap H=1$.
Conversely, the projection onto the complement of any semidirect
product is a retraction.

\begin{proposition}[Retract principle]\label{prop:retract-principle}
Let $H$ be a retract of a finitely generated group $G$. Then
\[
\Dio(H)\leq_P\Dio(G).
\]
The reduction regards a system with coefficients in $H$ as a system
with coefficients in $G$.
\end{proposition}

\begin{proof}
Let $\rho:G\to H$ be a retraction and let $\Sigma(X)=1$ have
coefficients in $H$. Every solution in $H$ is a solution in $G$.
Conversely, applying $\rho$ to a solution in $G$ gives a solution
in $H$, because homomorphisms commute with word evaluation and
$\rho$ fixes the coefficients.

The group $H$ is finitely generated, since it is an image of $G$.
After fixing generating sets, represent each generator of $H$ by
a fixed word in the generators of $G$. Substitution of these words
in the coefficient words of $\Sigma$ is a linear-time operation.
It produces a system solvable in $G$ exactly when $\Sigma$ is
solvable in $H$.
\end{proof}

\begin{corollary}[Undecidability ascends]
Let $G$ be finitely generated and $H$ a retract of $G$. If
$\Dio(H)$ is undecidable, then $\Dio(G)$ is undecidable.
\end{corollary}

\begin{corollary}[Decidability descends]
Let $G$ be finitely generated and $H$ a retract of $G$. If
$\Dio(G)$ is decidable, then $\Dio(H)$ is decidable.
\end{corollary}

\begin{corollary}[Semidirect products] \label{co:semidirect}
If $G=N\rtimes H$ is finitely generated, then
$\Dio(H)\leq_P\Dio(G)$. In particular, undecidability of $\Dio(H)$ implies undecidability of $\Dio(G)$.
\end{corollary}

\begin{proof}
The projection $N\rtimes H\to H$ is a retraction.
\end{proof}

\section{Indicable groups and retracts}
\label{se:indicable}

\subsection{Indicable groups}

\begin{definition}
A group $G$ is \emph{indicable} if it admits an epimorphism
$G\twoheadrightarrow\Z$. We denote the class of indicable groups
by $\mathcal I$ and its finitely generated members by
$\mathcal I_{\mathrm{fg}}$.
\end{definition}

\begin{proposition}\label{prop:indicable-characterization}
For an arbitrary group $G$, the following are equivalent:
\begin{enumerate}[label=\textup{(\roman*)}]
\item $G$ is indicable;
\item $G$ has an infinite cyclic subgroup that is a retract;
\item $G=K\rtimes\langle t\rangle$ for some normal subgroup $K$,
      where $\langle t\rangle\cong\Z$.
\end{enumerate}
If $G$ is finitely generated, these are also equivalent to
$\operatorname{rank}_{\Z}G_{\mathrm{ab}}>0$, where $G_{\mathrm{ab}}$ is the abelianization of $G$.
\end{proposition}

\begin{proof}
Given an epimorphism $\chi:G\to\Z$, choose $t$ with $\chi(t)=1$.
Then $t$ has infinite order, $g\mapsto t^{\chi(g)}$ is a retraction,
and $G=\ker\chi\rtimes\langle t\rangle$. Either a cyclic
retraction or the projection in the displayed semidirect product
gives an epimorphism onto $\Z$. Finally, every homomorphism to
$\Z$ factors through the abelianization. A finitely generated
abelian group maps onto $\Z$ precisely when its free abelian rank
is positive.
\end{proof}

\begin{proposition}
Each of the following groups is indicable:
\begin{enumerate}[label=\textup{(\arabic*)}]
\item Nontrivial free groups.
\item Free abelian groups $\Z^n$, $n\geq1$.
\item Semidirect products $K\rtimes\Z$.
\item HNN extensions; in particular, Baumslag--Solitar groups
      $BS(m,n)=\langle a,t\mid t^{-1}a^m t=a^n\rangle$,
      where $mn\neq0$.
\item Nontrivial finitely generated right-angled Artin groups.
\item Nontrivial finitely generated torsion-free nilpotent groups.
\item Nontrivial finitely generated free solvable groups.
\item Braid groups $B_n$, $n\geq2$.
\item Nontrivial torsion-free one-relator groups.
\item Knot groups and groups of nonempty links in $S^3$.
\item Groups having a finite presentation with positive deficiency.
\item Fundamental groups of finite connected graphs of groups whose
      underlying graphs have positive first Betti number.
\item Artin groups with nonempty standard generating sets.
\item Nontrivial residually free groups and nontrivial limit groups.
\item Thompson's group $F$.
\end{enumerate}
\end{proposition}

\begin{proof}
For a free group, send one basis element to $1\in\Z$ and all
others to $0$. 

Use a coordinate projection for $\Z^n$, and the
projection onto the complement for $K\rtimes\Z$. 

In an HNN
extension, send the base group to $0$ and the stable letter to
$1$; each defining HNN relation has stable-letter exponent sum
zero. This also proves the Baumslag--Solitar assertion.

For a right-angled Artin group, send a chosen vertex generator to
$1$ and all other generators to $0$; the commutator relations are
preserved. 

A nontrivial finitely generated torsion-free nilpotent
group is infinite and has infinite abelianization. To recall the
latter implication, if the abelianization of a finitely generated
nilpotent group were finite, each lower central factor would be
finite: it is a homomorphic image of a tensor power of the
abelianization under the map induced by iterated commutators.
The finite lower central series would then make the whole group
finite; see also~\cite{KMBook}. Thus the abelianization has positive rank. 

A nontrivial
free solvable group maps onto $\Z$ by sending one free generator
to $1$ and the others to $0$.

For $B_n$, the exponent-sum map sending each standard braid
generator to $1$ respects both the braid and commuting relations.

A one-relator presentation on $r\geq2$ generators has
abelianization of rank at least $r-1$. In the one-generator case,
a nontrivial torsion-free group given by such a presentation is
$\Z$. If the generating set is infinite, a generator not occurring
in the relator can be sent to $1$ and every other generator to $0$.

A link group with $k\geq1$ components has abelianization
$\Z^k$, generated by the meridians; the knot case is $k=1$.

If a presentation has $r$ generators and $s<r$ relators, its
abelianization is a quotient of $\Z^r$ by a subgroup generated by
at most $s$ elements and hence has positive rank. 

For a graph of
groups, killing all vertex groups gives an epimorphism onto the
fundamental group of the underlying graph. The latter is a
nontrivial free group when the first Betti number is positive.

For an Artin group, send every standard generator to $1$; the two
sides of each Artin relation have the same length.

For a nontrivial residually free group $G$, choose $g\neq1$ and a
homomorphism $\psi:G\to F$ to a free group with $\psi(g)\neq1$.
The image $\psi(G)$ is a nontrivial free group and therefore maps
onto $\Z$. Composing with $G\twoheadrightarrow\psi(G)$ proves
the assertion; nontrivial limit groups are a special case.

Finally, the standard presentation
\[
F=\langle x_0,x_1,x_2,\ldots\mid
x_i^{-1}x_jx_i=x_{j+1}\ (0\leq i<j)\rangle
\]
of Thompson's group admits the epimorphism sending $x_0$ to $1$
and all $x_j$, $j\geq1$, to $0$.
\end{proof}

Below we list a few more  general properties of indicable groups.

\begin{proposition}\label{prop:indicable-closure}
The following hold:
\begin{enumerate}[label=\textup{(\roman*)}]
\item Every finite-index subgroup of an indicable group is indicable.
\item A direct product or a free product is indicable if one factor
      is indicable.
\item A graph product is indicable if at least one vertex group is
      indicable.
\end{enumerate}
\end{proposition}

\begin{proof}
For \textup{(i)}, let $H\leq G$ have finite index and let
$\chi:G\twoheadrightarrow\Z$. The index
$[\Z:\chi(H)]$ is at most $[G:H]$, as follows by mapping a finite
set of coset representatives of $H$ to representatives of the
cosets of $\chi(H)$. Hence $\chi(H)=d\Z$ for some $d\geq1$;
compose $\chi|_H$ with the isomorphism $d\Z\cong\Z$.

For \textup{(ii)}, projection onto an indicable direct factor,
or killing all other free factors, gives a retraction onto that
factor. 

For \textup{(iii)}, killing every vertex group except a
chosen one respects all defining commuting relations of the
graph product and gives a retraction onto the chosen vertex
group. In both cases compose with its epimorphism onto $\Z$.
\end{proof}

\begin{remark}
The converse to Proposition~\ref{prop:indicable-closure}\textup{(i)}
fails. The infinite dihedral group has an infinite cyclic subgroup
of index two, but its abelianization is finite, so it is not
indicable.
\end{remark}
\section{Wreath products of groups}
\label{se:wreathproducts}

\subsection{General properties}
\label{ssec:wreath-conventions}

Let $B$ and $H$ be groups. Throughout the paper, $B\wr H$ denotes
their \emph{restricted regular wreath product}. Its base group is
\[
N=B^{(H)}=\{f:H\to B:\supp(f)\text{ is finite}\},
\qquad \supp(f)=\{x\in H:f(x)\ne1_B\},
\]
with pointwise multiplication. Thus $N$ is the restricted direct product
of the coordinate groups $B_x$, $x\in H$, rather than the unrestricted
direct product. The word \emph{regular} refers to the action of $H$ on
its own underlying set by left multiplication.

We use the right-exponent convention
\[
f^h(x)=f(hx)\qquad(f\in N,\ h,x\in H).
\]
It satisfies $(f^{h_1})^{h_2}=f^{h_1h_2}$. The wreath product is
$G=B\wr H=N\rtimes H$, but we write its elements in the unique
normal form $hf$, with the active factor first. Multiplication is
\begin{equation}\label{eq:1}
(h_1f_1)(h_2f_2)=h_1h_2 f_1^{h_2}f_2.
\end{equation}
In particular, the exponent agrees with conjugation:
$f^h=h^{-1}fh$. We identify $H$ with its canonical copy in $G$,
and $B$ with the coordinate subgroup at $1_H$. Thus $b\in B$ also
denotes the function with value $b$ at $1_H$ and value $1_B$ elsewhere.
For $b\ne1_B$, the support of $b^h$ is $\{h^{-1}\}$.
All commutators have the convention
\[
[x,y]=x^{-1}y^{-1}xy,
\qquad [x_1,\ldots,x_i]=[[x_1,\ldots,x_{i-1}],x_i].
\]

\begin{proposition}\label{prop1}
Let $G=B\wr H$ and $N=B^{(H)}$.
\begin{enumerate}[label=\textup{(\roman*)}]
\item $N$ is the normal closure of the canonical subgroup $B$ in $G$.
\item If $B$ is abelian, then $N$ is abelian and is a right
$\Z H$-module. In multiplicative notation its module operation is
\[
f^u=\prod_{j=1}^m(f^{n_j})^{h_j}
\quad\text{for}\quad u=\sum_{j=1}^m n_jh_j\in\Z H.
\]
\item If $B$ is abelian and $1\ne b\in B$, then $C_G(b)=N$.
\end{enumerate}
\end{proposition}

\begin{proof}
The base $N$ is normal, and the conjugates $B^h$, $h\in H$, are
exactly its coordinate subgroups. Every finitely supported function is
a finite product of coordinate functions, proving \textup{(i)}.

If $B$ is abelian, pointwise multiplication makes $N$ abelian. The
displayed formula extends the action of $H$ additively in the group-ring
variable. The identities
\[
f^{u+v}=f^u f^v,\qquad (fg)^u=f^u g^u,
\qquad (f^u)^v=f^{uv}
\]
follow from commutativity of $N$ and the right-action identity above;
these are precisely the module axioms in multiplicative notation.

For \textup{(iii)}, the base $N$ centralizes $b$. If $hf\in G$, then
$b^{hf}=(b^h)^f=b^h$, since $b^h,f\in N$. The supports of $b^h$
and $b$ are $\{h^{-1}\}$ and $\{1_H\}$, respectively. Thus
$b^{hf}=b$ if and only if $h=1_H$, proving $C_G(b)=N$.
\end{proof}

\subsection{Retracts in wreath products}

We retain the notation $G=B\wr H=N\rtimes H$ and the normal form $hf$
from Section~\ref{ssec:wreath-conventions}: $f\in B^{(H)}$ and
$f^h(x)=f(hx)$.  In particular,
\[
(h_1f_1)(h_2f_2)=h_1h_2f_1^{h_2}f_2.
\]

\begin{proposition}[Wreath products preserve indicability]
\label{prop:wreath-indicable}
Let $B$ and $H$ be arbitrary groups.  If either $B$ or $H$ is indicable,
then $B\wr H$ is indicable.
\end{proposition}

\begin{proof}
If $H$ is indicable, compose the projection $hf\mapsto h$ with an
epimorphism $H\twoheadrightarrow\Z$.

Suppose that $\chi:B\twoheadrightarrow\Z$ is an epimorphism.  For
$f\in B^{(H)}$, define
\[
\sigma(f)=\sum_{x\in H}\chi(f(x)).
\]
This finite sum defines a homomorphism $B^{(H)}\to(\Z,+)$, and
$\sigma(f^h)=\sigma(f)$ for every $h\in H$, since $x\mapsto hx$
permutes $H$.  Consequently the map
\[
\widehat\chi:B\wr H\longrightarrow\Z,
\qquad \widehat\chi(hf)=\sigma(f),
\]
is a homomorphism: indeed,
\[
\widehat\chi((h_1f_1)(h_2f_2))
=\sigma(f_1^{h_2}f_2)=\sigma(f_1)+\sigma(f_2).
\]
Choose $b\in B$ with $\chi(b)=1$.  The function equal to $b$ at $1_H$
and to $1_B$ elsewhere has image $1$ under $\sigma$, so
$\widehat\chi$ is surjective.
\end{proof}

\begin{theorem}\label{retract}
Let $B$ and $H$ be arbitrary groups.  If $C\leq B$ is an abelian retract
of $B$ and $D\leq H$ is a retract of $H$, then $B\wr H$ has a retract
canonically isomorphic to $C\wr D$.
\end{theorem}

\begin{proof}
Fix retractions $p:B\to C$ and $q:H\to D$.  We construct homomorphisms
\[
C\wr D\xrightarrow{\ J\ }B\wr H
\xrightarrow{\ R\ }C\wr D
\quad\text{such that}\quad R\circ J=\operatorname{id}_{C\wr D}.
\]

For $f\in B^{(H)}$, put
\begin{equation}\label{eq:P-definition}
P(f)(z)=\prod_{\substack{x\in H\\q(x)=z}}p(f(x)),
\qquad z\in D.
\end{equation}
Only finitely many factors are nontrivial, and their order does not matter
because $C$ is abelian.  Moreover,
$\supp(P(f))\subseteq q(\supp(f))$, so $P(f)\in C^{(D)}$.
For $f,g\in B^{(H)}$, commutativity in $C$ gives
\[
P(fg)(z)
=\prod_{q(x)=z}p(f(x))p(g(x))
=P(f)(z)P(g)(z).
\]
Thus $P:B^{(H)}\to C^{(D)}$ is a homomorphism.

The actions on both base groups use the same right-exponent convention.
For $h\in H$ and $z\in D$, the substitution $y=hx$ gives
\begin{align*}
P(f^h)(z)
&=\prod_{q(x)=z}p(f(hx))\\
&=\prod_{q(y)=q(h)z}p(f(y))\\
&=P(f)(q(h)z)=P(f)^{q(h)}(z).
\end{align*}
Hence
\begin{equation}\label{eq:equivariance}
P(f^h)=P(f)^{q(h)}.
\end{equation}
Define
\[
R:B\wr H\longrightarrow C\wr D,
\qquad R(hf)=q(h)P(f).
\]
Using the normal-form multiplication in the two wreath products, we obtain
\begin{align*}
R((h_1f_1)(h_2f_2))
&=q(h_1)q(h_2)P(f_1^{h_2}f_2)\\
&=q(h_1)q(h_2)P(f_1)^{q(h_2)}P(f_2)\\
&=(q(h_1)P(f_1))(q(h_2)P(f_2)).
\end{align*}
Therefore $R$ is a homomorphism.

For $u\in C^{(D)}$, extend $u$ to a function $\widetilde u\in B^{(H)}$
by setting
\begin{equation}\label{eq:extension}
\widetilde u(x)=
\begin{cases}
u(x),&x\in D,\\
1_B,&x\notin D.
\end{cases}
\end{equation}
The extension map is an injective homomorphism and satisfies
$\widetilde{u^d}=(\widetilde u)^d$ for every $d\in D$: multiplication
by $d$ on the left preserves both $D$ and $H\setminus D$.
Consequently
\[
J:C\wr D\longrightarrow B\wr H,
\qquad J(du)=d\widetilde u,
\]
is an injective homomorphism.  Explicitly,
\[
J((d_1u_1)(d_2u_2))
=d_1d_2\widetilde{u_1^{d_2}u_2}
=d_1d_2(\widetilde u_1)^{d_2}\widetilde u_2
=J(d_1u_1)J(d_2u_2).
\]

Finally, $q(d)=d$ for $d\in D$.  In the product defining
$P(\widetilde u)(z)$, every factor indexed by $x\notin D$ is trivial;
among the indices in $D$, the equality $q(x)=z$ holds only at $x=z$.
Since $p$ fixes $C$ pointwise, it follows that
\[
P(\widetilde u)(z)=p(u(z))=u(z).
\]
Thus
\[
R\circ J=\operatorname{id}_{C\wr D}.
\]
In particular, $J$ identifies $C\wr D$ with a subgroup of $B\wr H$,
and the homomorphism
\[
J\circ R:B\wr H\longrightarrow J(C\wr D)
\]
fixes that subgroup pointwise. Indeed, for every $v\in C\wr D$,
\[
(J\circ R)(J(v))=J\bigl((R\circ J)(v)\bigr)=J(v).
\]
Therefore $J(C\wr D)$ is a retract of $B\wr H$, as required.
\end{proof}

\begin{remark}
The abelian hypothesis on $C$ is used when multiplying over the fibers
of $q$. In general, this construction cannot be used with nonabelian $C$:
different coordinates of $B^{(H)}$ commute, and folding two such
coordinates into a single coordinate would require their images to
commute.  No commutativity assumption on $D$ is needed.
\end{remark}

\begin{corollary}
    Let $B$ and  $H$ be arbitrary groups. Then $H$ is a retract of $B \wr H$.
\end{corollary}
\begin{proof}
It follows from Theorem \ref{retract} since the identity subgroup is an abelian retract of $B$, and $H$ is a retract of $H$.
\end{proof}

\begin{proposition}\label{prop:indicable-wreath-retract}
If $B,H\in\mathcal I$, then $B\wr H$ has a retract isomorphic to
$\Z\wr\Z$.
\end{proposition}

\begin{proof}
By Proposition~\ref{prop:indicable-characterization}, there are infinite
cyclic retracts $C\leq B$ and $D\leq H$.  The group $C$ is abelian, so
Theorem~\ref{retract} gives a retract isomorphic to
$C\wr D\cong\Z\wr\Z$.
\end{proof}

\subsection{Retracts and the Diophantine problem in wreath products}

For the algorithmic assertions below, groups are finitely generated and
coefficients are represented by words in fixed finite generating sets.
If $B$ and $H$ are finitely generated, so is $B\wr H$.  In case of finitely generated groups decidability of the Diophantine problem is independent of a chosen finite set of generators.

\begin{theorem}\label{thm:main}
Let $B$ and $H$ be arbitrary finitely generated groups, with $B$
passive and $H$ active, and put $G=B\wr H$. Then
\[
\boxed{
\Dio(B)\le_P\Dio(G)
\qquad\text{and}\qquad
\Dio(H)\le_P\Dio(G).
}
\]
In particular, if the Diophantine problem in either $B$ or $H$
is undecidable, then the Diophantine problem in $G$ is undecidable.
\end{theorem}

\begin{proof}
Write
\[
G=N\rtimes H,\qquad N=\bigoplus_{h\in H}B_h,
\]
where $B_h\cong B$ is the coordinate subgroup indexed by $h$.
Thus $N$ consists of the finitely supported functions
$f\colon H\to B$, with pointwise multiplication. Set
\[
(\alpha_h f)(x)=f(h^{-1}x)\qquad(h,x\in H).
\]
For elements of $B \wr H$ expressed as pairs in $B\times H$, the multiplication and inverse formulas are
\begin{align*}
(f,h)(g,k)&=(f\,\alpha_{h^{-1}}(g),hk),\\
(f,h)^{-1}&=(\alpha_{h}(f^{-1}),h^{-1}),
\end{align*}
where $f^{-1}(x)=f(x)^{-1}$. Let
$\iota(h)=(1,h)$, and let $\varepsilon_h(b)\in N$ take the value
$b$ at $h$ and the value $1$ elsewhere. The maps
$\iota\colon H\hookrightarrow G$ and
$\varepsilon_1\colon B\hookrightarrow G$ are the canonical
embeddings of the factors.

\medskip
\noindent\emph{Reduction from the active group $H$ to $G$.}
The canonical projection
\[
\pi\colon G\longrightarrow H,\qquad \pi(f,h)=h
\]
satisfies $\pi\circ\iota=\operatorname{id}_H$. Equivalently,
$\iota\circ\pi$ is a retraction of $G$ onto $\iota(H)$. By Proposition \ref{prop:retract-principle} $\Dio(H)\le_P\Dio(G)$, as required.

\medskip
\noindent\emph{Reduction from the passive group $B$ to $G$.}
If $H=1$, then $G=B$ and the identity map is a
reduction. If $B=1$, every system over $B$ has a solution, so
every instance can be sent to the solvable equation $X=1$ in $G$.
Hence assume that $B\ne1$ and $H\ne1$.

Choose $b_0\in B\setminus\{1\}$ and $h_0\in H\setminus\{1\}$, and put
\begin{equation}\label{eq:u}
u=\varepsilon_{h_0}(b_0)
 =\iota(h_0)\varepsilon_1(b_0)\iota(h_0)^{-1}.
\end{equation}
We claim that
\begin{equation}\label{eq:centralizer}
C:=C_G(u)
 =\varepsilon_{h_0}\bigl(C_B(b_0)\bigr)
   \times\bigoplus_{h\ne h_0}B_h.
\end{equation}
Indeed, if $q=(f,k)\in G$, then $q^{-1}uq$ has value
\[
f(x)^{-1}\,u(kx)\,f(x)
\]
at $x\in H$. This value is nontrivial precisely when $x=k^{-1}h_0$.
Thus $q^{-1}uq=u$ forces $k^{-1}h_0=h_0$, and hence $k=1$.
For $f\in N$, commuting with $u$ is equivalent to
$f(h_0)\in C_B(b_0)$; all other coordinates are unrestricted.
This proves \eqref{eq:centralizer}.

Since $h_0\ne1$, the subgroup $B_1$ is a direct factor of $C$.
The coordinate projection
\[
r\colon C\longrightarrow B,\qquad r(f)=f(1)
\]
satisfies $r\circ\varepsilon_1=\operatorname{id}_B$, so
$\varepsilon_1\circ r$ is a retraction onto $B_1$.
Moreover, $C$ is defined in $G$ by the single equation
\begin{equation}\label{eq:domain}
[X,u]=1,
\end{equation}
where $[x,y]=x^{-1}y^{-1}xy$.

Let $\mathbf X=(X_1,\ldots,X_n)$ and
$\mathbf b=(b_1,\ldots,b_s)$. Given a finite system over $B$,
\[
\Sigma:\qquad w_j(\mathbf X;\mathbf b)=1
\qquad(1\le j\le m),
\]
form the following system over $G$:
\begin{equation}\label{eq:translated}
\widehat\Sigma:\qquad
\begin{cases}
w_j(\mathbf X;\varepsilon_1(b_1),\ldots,\varepsilon_1(b_s))=1,
 &1\le j\le m,\\
[X_i,u]=1,
 &1\le i\le n.
\end{cases}
\end{equation}
A solution of $\Sigma$ in $B$ gives a solution of
$\widehat\Sigma$ by applying $\varepsilon_1$ to its entries:
elements supported at $1$ commute with $u$, which is supported at
$h_0\ne1$. Conversely, every entry of a solution of
$\widehat\Sigma$ lies in $C$. Applying $r$ gives a solution of
$\Sigma$ in $B$, since $r(\varepsilon_1(b_j))=b_j$. Therefore
\[
\Sigma\text{ has a solution in }B
\quad\Longleftrightarrow\quad
\widehat\Sigma\text{ has a solution in }G.
\]

This transformation is computable in polynomial time. For fixed finite generating
sets $S_B$ and $S_H$ of the factors,
$\varepsilon_1(S_B)\cup\iota(S_H)$ generates $G$. Coefficient words
over $S_B$ are simply viewed as words over $\varepsilon_1(S_B)$.
Fix words representing $b_0$ and $h_0$ once and for all. The
formula (\ref{eq:u}) for $u$ then gives a fixed word in the
generators of $G$, so adjoining \eqref{eq:domain} for each unknown
is computable in polynomial time.  This proves $\Dio(B)\le_P\Dio(G)$ and completes the proof.
\end{proof}

\begin{corollary}[Finite active group]\label{cor:finite}
Let $B$ be a finitely generated group and let $H$ be a finite group.
Then
\[
\boxed{\Dio(B\wr H)\text{ is decidable}
\quad\Longleftrightarrow\quad
\Dio(B)\text{ is decidable}.}
\]
\end{corollary}

\begin{proof}
Put $G=B\wr H$. If $\Dio(G)$ is decidable,
Theorem~\ref{thm:main} gives the decidability of $\Dio(B)$.

Suppose, conversely, that $\Dio(B)$ is decidable. In our case,
\[
G=N\rtimes H,\qquad N=\bigoplus_{h\in H}B_h,
\]
where $B_h\cong B$ is the coordinate subgroup indexed by $h$. Since $H$ is finite, $N$ is a finite direct sum of groups with decidable Diophantine problems. Now the result follows from \cite{CiobanuHoltRees2020},  Lemma 3.6.

\end{proof}

\begin{theorem}\label{cor:wreath-diophantine-reduction}
If $B,H\in\mathcal I_{\mathrm{fg}}$, then
 the Diophantine problem in $B\wr H$ is undecidable.
\end{theorem}

\begin{proof}
By Proposition~\ref{prop:indicable-wreath-retract}, the group $B\wr H$
has a retract isomorphic to $\Z\wr\Z$.  Apply
Proposition~\ref{prop:retract-principle} and the known undecidability of
the Diophantine problem in $\Z\wr\Z$~\cite{MR}; see also the
discussion in \cite{KM2025}.
\end{proof}

\begin{corollary}
Let $B,H$ be finitely generated infinite abelian groups.  Then $\Dio(B\wr H)$ is undecidable. 
\end{corollary}
It was known for $B,H$  free abelian of finite rank $\geq 1$ \cite{KM2025}.

\begin{corollary}
Let $B,H$ be finitely generated groups.  If $C$ is an abelian retract of
$B$ and $D$ is a retract of $H$, then
\[
\Dio(C\wr D)\leq_P\Dio(B\wr H).
\]
\end{corollary}

\begin{proof}
The groups $C,D$ are finitely generated, being images of $B,H$.
The conclusion follows from Theorem~\ref{retract} and
Proposition~\ref{prop:retract-principle}.
\end{proof}

\begin{corollary}\label{semi}
Let $C$ be abelian.  Suppose that the semidirect products
$B=K_1\rtimes C$ and $H=K_2\rtimes D$ are finitely generated, where
$K_1,K_2$ are the normal factors.  Then
\[
\Dio(C\wr D)
\leq_P\Dio\bigl((K_1\rtimes C)\wr(K_2\rtimes D)\bigr).
\]
\end{corollary}

\begin{proof}
The canonical projections onto the complementary factors are retractions
$p:B\to C$ and $q:H\to D$.  Apply the preceding corollary.
\end{proof}

\section{Wreath products with abelian base groups}
\label{sec:abelian-base}

 Throughout this
section, we use the conventions of Section~\ref{ssec:wreath-conventions}:
\[
G=B\wr H=N\rtimes H,\qquad N=B^{(H)},\qquad
f^h(s)=f(hs),\qquad
(h_1f_1)(h_2f_2)=h_1h_2f_1^{h_2}f_2.
\]

\subsection{Free abelian bases and lower central quotients}

Let $B$ be a free abelian group with basis $b_1,\ldots,b_n$,
where $n\geq1$, and $H$ an arbitrary finitely generated group with a free abelian subgroup $A\leq H$  with basis
$a_1,\ldots,a_k$, where $k\geq1$.  

As usual, each $b_j$ denotes the function supported at
$1\in H$ with value $b_j$ there.  The group $N$ is a free right
$\mathbb ZH$-module with basis $b_1,\ldots,b_n$, written multiplicatively.
For example, if $p=\sum_h m_hh\in\mathbb ZH$, then
\[
b_j^p=\prod_h (b_j^{m_h})^h.
\]
The factors commute because $N$ is abelian, and the representation
$f=\prod_{j=1}^n b_j^{p_j}$ is unique.  Indeed, $b_j^h$ is supported at
$h^{-1}$, so this statement follows by reading the exponent of $b_j$ at
each coordinate of $H$.

  Set
\[
R=\mathbb ZA=\mathbb Z[a_1^{\pm1},\ldots,a_k^{\pm1}],\qquad
M=\{f\in N:\supp(f)\subseteq A\},\qquad T=\langle A,B\rangle.
\]
Thus $M$ is the free $R$-module with basis $b_1,\ldots,b_n$; it is
generally a proper submodule of $N$.  Write
\[
\epsilon:R\longrightarrow\mathbb Z,\qquad
\epsilon\left(\sum_{a\in A}m_aa\right)=\sum_{a\in A}m_a,
\qquad \Delta=\ker\epsilon.
\]
The augmentation ideal is
$\Delta=(a_1-1,\ldots,a_k-1)$.  This follows from
$uv-1=(u-1)v+(v-1)$ and
$a_i^{-1}-1=-a_i^{-1}(a_i-1)$.
For an ideal $I\unlhd R$, denote by $M^I$ the subgroup generated by
the elements $f^p$ with $f\in M$ and $p\in I$.  Equivalently,
\[
M^I=\left\{\prod_{j=1}^n b_j^{p_j}:p_1,\ldots,p_n\in I\right\}.
\]

\begin{lemma}\label{computationFreeAbelian}
With this notation,
\begin{enumerate}[label=\textup{(\roman*)}]
\item $T=M\rtimes A$, and restriction of functions from $H$ to $A$
identifies $T$ with $B\wr A$.
\item For every $i\geq2$,
\[
\gamma_i(T)=M^{\Delta^{i-1}}.
\]
\item For every $i\geq2$,
\[
T/\gamma_i(T)\cong (M/M^{\Delta^{i-1}})\rtimes A.
\]
\end{enumerate}
\end{lemma}

\begin{proof}
Conjugating $B$ by elements of $A$ gives exactly the copies of $B$
supported at the points of $A$.  These copies generate $M$.
The subgroup $M$ is invariant under $A$, and $M\cap A=1$.
It follows that $\langle A,B\rangle=M\rtimes A$.  The action on $M$
is the same regular action used to define $B\wr A$, proving the first
assertion.

For the commutator convention $[x,y]=x^{-1}y^{-1}xy$, the multiplication
rule gives, for $\alpha,\beta\in A$ and $f,g\in M$,
\begin{equation}\label{eq:mixedcommutator}
[\alpha f,\beta g]=f^{-1}g^{-\alpha}f^\beta g
                  =f^{\beta-1}g^{1-\alpha}.
\end{equation}
Here $g^{-\alpha}$ means $(g^{-1})^\alpha$.
In particular, if $f=\prod_jb_j^{p_j}$ and
$g=\prod_jb_j^{q_j}$, then
\begin{equation}\label{commutator_in_T}
[\alpha f,\beta g]
=\prod_{j=1}^n b_j^{\,p_j(\beta-1)+q_j(1-\alpha)}.
\end{equation}
It follows that $[T,T]\leq M^\Delta$.  Conversely,
$[f,a_j]=f^{a_j-1}$ for every $f\in M$, and the elements
$a_j-1$ generate $\Delta$ as an ideal.  Hence
$[T,T]=M^\Delta$.

More generally, since $M$ is abelian and $A$ preserves every
$M^{\Delta^r}$, we have
\[
[M^{\Delta^r},T]=[M^{\Delta^r},A]=M^{\Delta^{r+1}}
\qquad(r\geq0).
\]
The inclusion from left to right follows from
$[u,\alpha]=u^{\alpha-1}$; the reverse inclusion follows because
$\Delta^{r+1}$ is generated by the products $p(a_j-1)$ with
$p\in\Delta^r$.  Induction proves the second assertion.  The associated
iterated-commutator identity is
\begin{equation}\label{commutatorAssertion}
[f,a_{j_1},\ldots,a_{j_r}]
=f^{(a_{j_1}-1)\cdots(a_{j_r}-1)}
\qquad(f\in M, r\geq1).
\end{equation}
All commutators here are left-normed.

Finally, $M^{\Delta^{i-1}}$ is $A$-invariant, so the map
\[
\alpha f\longmapsto\alpha\bigl(fM^{\Delta^{i-1}}\bigr)
\]
is a surjective homomorphism from $T$ to
$(M/M^{\Delta^{i-1}})\rtimes A$.  Its kernel is precisely
$M^{\Delta^{i-1}}=\gamma_i(T)$, proving the third assertion.
\end{proof}

\begin{lemma}\label{lem:lower-central-explicit}
For every fixed $i\geq2$, the subgroup $\gamma_i(T)$ is e-definable
in $T$.
\end{lemma}

\begin{proof}

By Lemma \ref{computationFreeAbelian} bullet (ii), $[T,T]$ is the abelian group $M^{\Delta}$. Thus, $T$ is a finitely generated metabelian group, which is verbally elliptic (see section \ref{sse:Verbal}). Consequently, all of its subgroups $\gamma_i(T)$ are e-definable in $T$ for $i \geq 2$.
\end{proof}

\begin{lemma}\label{notVA}
The finitely generated nilpotent group $Q=T/\gamma_3(T)$ is not
virtually abelian.
\end{lemma}

\begin{proof}
Recall that
\[
R=\mathbb ZA=\mathbb Z[a_1^{\pm1},\ldots,a_k^{\pm1}],
\qquad
\Delta=(a_1-1,\ldots,a_k-1),
\]
and that $M$ is a free $R$-module with basis $b_1,\ldots,b_n$.
By the preceding lower-central-series calculation,
\[
\gamma_2(T)=M^\Delta,
\qquad
\gamma_3(T)=M^{\Delta^2}.
\]
By definition  $Q$ has nilpotency class at most two.

Let $\bar a_1,\bar b_1$ denote the images of $a_1,b_1$ in $Q$.
We first show that
\[
c=[\bar b_1,\bar a_1]
=b_1^{\,a_1-1}M^{\Delta^2}
\]
has infinite order.

Suppose that $c^m=1$ for some $m\in\mathbb Z$. Since $M$ is free
over $R$ with basis $b_1,\ldots,b_n$, this implies
\[
m(a_1-1)\in\Delta^2.
\]
Consider the ring homomorphism
\[
\sigma:R\longrightarrow\mathbb Z[u,u^{-1}]
\]
defined by
\[
\sigma(a_1)=u,\qquad
\sigma(a_i)=1\quad(2\leq i\leq k).
\]
It sends $\Delta$ onto the ideal $(u-1)$. Therefore
\[
m(u-1)\in(u-1)^2\mathbb Z[u,u^{-1}],
\]
so there exists $h(u)\in\mathbb Z[u,u^{-1}]$ such that
\[
m(u-1)=(u-1)^2h(u).
\]
The Laurent polynomial ring $\mathbb Z[u,u^{-1}]$ is an integral
domain. Cancelling $u-1$ and then evaluating at $u=1$, we obtain
\[
m=(u-1)h(u),
\qquad\text{hence}\qquad m=0.
\]
Thus $c$ has infinite order.

Now suppose, towards a contradiction, that $Q$ contains an abelian
subgroup $U$ of finite index. Every element of $Q$ has a positive
power in $U$. Hence there exist positive integers $r,s$ such that
\[
\bar a_1^{\,r}\in U,
\qquad
\bar b_1^{\,s}\in U.
\]
Since $U$ is abelian,
\[
[\bar b_1^{\,s},\bar a_1^{\,r}]=1.
\]
On the other hand, commutators are central in $Q$, so
\[
[\bar b_1^{\,s},\bar a_1^{\,r}]
=
[\bar b_1,\bar a_1]^{sr}
=
c^{sr}.
\]
This is nontrivial because $c$ has infinite order and $sr>0$, a
contradiction.
Therefore,  $Q$ has no abelian subgroup of finite index and is not
virtually abelian.
\end{proof}

The preceding two lemmas give an alternative proof of undecidability
for $T\cong B\wr A$ when $B$ and $A$ are nontrivial free abelian groups
of finite rank: apply Theorem~\ref{thm:nilpotent-criterion} to the
e-definable subgroup $\gamma_3(T)$ and the non-virtually abelian
nilpotent quotient $T/\gamma_3(T)$.

\subsection{Transferring undecidability to the ambient wreath product}

We now allow torsion in both abelian groups. Let $B$ be an infinite
finitely generated abelian group, let $H$ be finitely generated, and
let $A\leq H$ be an infinite finitely generated abelian subgroup that
is \emph{e-definable in $H$}.

Choose generating sets
$B=\langle b_1,\ldots,b_n\rangle$ and
$A=\langle a_1,\ldots,a_k\rangle$ with $b_1\ne1$ and $a_1$ of infinite
order. Put $G=B\wr H$, $N=B^{(H)}$,
\[
M=\{f\in N:\operatorname{supp}(f)\subseteq A\},
\qquad T=\langle A,B\rangle.
\]
Here $B$ is embedded at the identity coordinate, as in
Section~\ref{ssec:wreath-conventions}. Conjugating this copy of $B$
by $A$ generates $M$, so $T=M\rtimes A\cong B\wr A$.
By Proposition \ref{prop1}  $C_G(b_1)=N$. 

Fix an e-defining formula for $A$ in $H$,
\[
\phi(x)=\exists\bar y\bigwedge_{\ell=1}^q
w_\ell(x,\bar y,\bar h)=1,
\]
where $\bar h$ is a fixed tuple of coefficients in $H$.
Let $\pi:G\to H$ be the canonical retraction $\pi(hf)=h$ and put
$K=\pi^{-1}(A)=N\rtimes A$.

\begin{lemma}\label{definableInG}
Both $K$ and the embedded subgroup $A$ are e-definable in $G$.
\end{lemma}

\begin{proof}
The equality $C_G(b_1)=N$ allows us to translate equations in $G/N$
into equations in $G$.  Specifically, define
\[
\kappa(x)=\exists\bar y\bigwedge_{\ell=1}^q
[w_\ell(x,\bar y,\bar h),b_1]=1.
\]
We use the elements $\bar h\in H\leq G$ as coefficients.
For each word $w_\ell$, the equation $[w_\ell(x,\bar y,\bar h),b_1]=1$ 
precisely states that $w_\ell(x,\bar y,\bar h)\in N$, or equivalently
\[
w_\ell(\pi(x),\pi(\bar y),\bar h)=1\quad\text{in }H.
\]
Every tuple in $H$ has a lift to $G$.  Hence
$G\models\kappa(x)$ if and only if $H\models\phi(\pi(x))$,
which is equivalent to $x\in K$.

The formula
\[
\phi_G(x)=\kappa(x)\land[x,a_1]=1
\]
defines $A$ itself.  To see this, write $x=\alpha f\in K$ with
$\alpha\in A$ and $f\in N$.  Since $A$ is abelian,
$[x,a_1]=[f,a_1]$.  Thus $[x,a_1]=1$ if and only if
$f^{a_1}=f$.  The element $a_1$ has infinite order, so every orbit
$\{a_1^rs:r\in\mathbb Z\}$ is infinite.  Finite support forces
$f=1$.  Conversely, every element of $A$ satisfies both conjuncts.
\end{proof}

\begin{lemma}\label{solveTinK}
The subgroup $T=M\rtimes A\cong B\wr A$ is a retract of
$K=N\rtimes A$.
Consequently, every finite system of equations with coefficients in
$T$ has a solution in $K$ if and only if it has a solution in $T$.
\end{lemma}

\begin{proof}
Define $\lambda:N\to M$ by
\[
\lambda(f)(s)=
\begin{cases}
f(s),&s\in A,\\
1,&s\notin A.
\end{cases}
\]
This is a homomorphism for pointwise multiplication and is the
identity on $M$.  For $\alpha\in A$, the equivalence
$\alpha s\in A\iff s\in A$ gives
\[
\lambda(f^\alpha)=\lambda(f)^\alpha.
\]
Therefore, the map
\[
\rho:K\longrightarrow T,\qquad \rho(\alpha f)=\alpha\lambda(f)
\]
is a homomorphism: 
\[
\rho((\alpha f)(\beta g))
 =\alpha\beta\lambda(f^\beta g)
  =\alpha\beta\lambda(f)^\beta\lambda(g)
 =\rho(\alpha f)\rho(\beta g).
\]
It fixes every element of $T$, so it is a retraction.
A solution in $K$ of a system with coefficients in $T$ maps under
$\rho$ to a solution in $T$, because $\rho$ fixes the coefficients.
Conversely, a solution in $T$ is already a solution in $K$.
\end{proof}

We also record a simple fact used in the examples below.

\begin{lemma} \label{le:e-def-abelian}
Let $A\cong\Z^n\times F$ be an infinite finitely generated abelian
group, where $n\geq1$ and $F$ is finite. Then $A$ contains an
e-definable subgroup isomorphic to $\Z^n$. More precisely, if
$d=|F|$, the subgroup $A^d=\{y^d:y\in A\}$ is free abelian of rank
$n$, has finite index in $A$, and is e-definable in $A$.
\end{lemma}
\begin{proof}
Since $A$ is abelian, the power map $\psi: y\mapsto y^d$ is a
homomorphism. Hence, $\psi(A) = \psi(\Z^n\times F) = d\Z^n\times\{1\}$, which is free abelian of rank $n$ and has
finite index. The formula
\[
x\in A^d\quad\Longleftrightarrow\quad
\exists y\;(x=y^d)
\]
defines this subgroup by an equation in $A$.
This proves the lemma. 
\end{proof}

\begin{theorem}\label{main}
Let $B$ be an infinite finitely generated abelian group, and
let $H$ be a finitely generated group containing an infinite 
e-definable finitely generated abelian subgroup $A$.  Then the
Diophantine problem in $G=B\wr H$ is undecidable.
\end{theorem}

\begin{proof}
Let $N=B^{(H)}$ and let $\pi:G\to H$ be the canonical retraction.
By Lemma~\ref{definableInG}, the subgroup
\[
K=\pi^{-1}(A)=N\rtimes A
\]
is e-definable in $G$. Put
\[
M=\{f\in N:\supp(f)\subseteq A\},
\qquad T=M\rtimes A\cong B\wr A.
\]
We recall explicitly the retraction from Lemma~\ref{solveTinK}:
\[
\boxed{\rho:K\twoheadrightarrow T\cong B\wr A,\qquad
       \rho(\alpha f)=\alpha\lambda(f),}
\]
where $\alpha\in A$, $f\in N$, and
\[
\lambda(f)(s)=
\begin{cases}
f(s),&s\in A,\\
1,&s\notin A.
\end{cases}
\]
Thus $\lambda$ restricts a base function to $A$ and extends it by
the identity outside $A$. It is a homomorphism and is
$A$-equivariant: for $\beta\in A$, the equivalence
$\beta s\in A\iff s\in A$ gives
$\lambda(f^\beta)=\lambda(f)^\beta$. This yields
\[
\rho((\alpha f)(\beta g))
 =\alpha\beta\lambda(f^\beta g)
 =\alpha\beta\lambda(f)^\beta\lambda(g)
 =\rho(\alpha f)\rho(\beta g).
\]
Moreover, $\lambda(f)=f$ for $f\in M$, so $\rho$ fixes $T$
pointwise. This proves that $\rho$ is a retraction.

Both $A$ and $B$ are infinite finitely generated abelian groups,
so they are indicable. Theorem~\ref{cor:wreath-diophantine-reduction}
therefore gives undecidability of $\Dio(T)=\Dio(B\wr A)$.
For a finite system with coefficients in $T$, constrain its variables
to the e-definable subgroup $K$. Such a system has a solution in $K$
if and only if it has a solution in $T$: apply $\rho$ to a solution
in $K$, using that $\rho$ fixes every coefficient.
Expanding the defining equations for $K$ gives a system in $G$.
Because $K$ is e-definable in $G$, we have that $K$ is e-interpretable in $G$ - so by Lemma \ref{lem:interpretation-reduction} this process provides a polynomial-time reduction $\Dio(T) \leq_P \Dio(G)$. Consequently, $\Dio(G)$ is undecidable. 
\end{proof}

\subsection{Groups with e-definable nontrivial free abelian subgroups of finite rank}
\label{se:ED}

Write ${\mathcal ED}$ for the class of groups containing an
e-definable nontrivial free abelian subgroup of finite rank.
Parameters are allowed in the defining equations. We give several
families of groups in this class.

The first type of groups in ${\mathcal ED}$ is related to abelian centralizers.  
For a subset $S\subseteq G$, write
$\Cent_G(S)=\{x\in G \mid [x,s]=1\text{ for every }s\in S\}$.

\begin{proposition}[The maximal abelian subgroup criterion]\label{thm:maximal}
Let $M=\langle m_1,\ldots,m_k\rangle$ be a nontrivial finitely generated
maximal abelian subgroup of $G$. Then
\begin{equation}\label{eq:maximal-definition}
M=\Cent_G(M)=\Cent_G(m_1,\ldots,m_k)
 =\{x\in G:[x,m_i]=1\ (1\leq i\leq k)\}.
\end{equation}
Thus $M$ is e-definable with parameters by a finite system of equations
without auxiliary variables.

If $M$ is infinite, it contains an e-definable nontrivial free abelian
subgroup of finite rank, which is also e-definable in $G$.
\end{proposition}

\begin{proof}
An element $x$ commutes with all the $m_i$ if and only if it centralizes
$M$. In that event, $\langle M,x\rangle$ is abelian, so the maximality of
$M$ implies $x\in M$. The reverse inclusion holds because $M$ is abelian.
This proves~\eqref{eq:maximal-definition}.

If $M$ is infinite, let $d$ be the order of its torsion subgroup.
By Lemma~\ref{le:e-def-abelian}, $M^d$ is a nontrivial free abelian
subgroup of finite rank, e-definable in $M$. Transitivity of
e-definability makes $M^d$ e-definable in $G$.
\end{proof}

\begin{proposition}
Let $H$ be an infinite polycyclic group. Then $H \in {\mathcal ED}$.
\end{proposition}
\begin{proof}
Let $H$ be an infinite polycyclic group, and choose a polycyclic
series
\[
H=H_0\trianglerighteq H_1\trianglerighteq\cdots
\trianglerighteq H_r=1,
\]
where each factor $H_i/H_{i+1}$ is cyclic. If all these factors
were finite, then $H$ would be finite. Hence some factor
$H_i/H_{i+1}$ is infinite cyclic.

Choose $g\in H_i$ whose image generates $H_i/H_{i+1}$.
Since that image has infinite order, $g$ has infinite order.
Let $A$ be a maximal abelian subgroup of $H$ containing $g$;
it exists by Zorn's lemma. Every subgroup of a polycyclic group
is finitely generated, so $A$ is finitely generated and infinite.
By Proposition~\ref{thm:maximal}, $A$ is e-definable in $H$.
If $d$ is the order of the torsion subgroup of $A$, then $A^d$
is nontrivial free abelian of finite rank and e-definable in $H$,
by Lemma~\ref{le:e-def-abelian}. Thus $H\in{\mathcal ED}$.

\end{proof}

\begin{proposition}
Let $H$ be a finitely generated group whose center is infinite and finitely
generated. Then $H \in {\mathcal ED}$.
\end{proposition}
\begin{proof}
If $H=\langle h_1,\ldots,h_s\rangle$, then
\[
Z(H)=\{z\in H:[z,h_i]=1\ (1\leq i\leq s)\}.
\]
Thus $Z(H)$ is e-definable in $H$. Let $d$ be the order of the
torsion subgroup of $Z(H)$. By Lemma~\ref{le:e-def-abelian},
$Z(H)^d$ is nontrivial free abelian of finite rank and e-definable
in $Z(H)$, hence in $H$. Therefore $H\in{\mathcal ED}$.
\end{proof}

\begin{proposition}
Let $H$ be an infinite finitely generated virtually abelian group. Then $H \in {\mathcal ED}$.
\end{proposition}

\begin{proof}
Choose an abelian subgroup \(A\leq H\) of finite index. Among all abelian
subgroups containing \(A\), choose \(M\) for which \([M:A]\) is largest.
This is possible because these indices are positive integers bounded
by \([H:A]\). Then \(M\) is a maximal abelian subgroup of \(H\):
a proper abelian overgroup would have a strictly larger index over \(A\).

Since \(M\) has finite index in the finitely generated group \(H\),
Schreier's theorem implies that \(M\) is finitely generated.
It is infinite because \(H\) is infinite. Write
\(M=\langle m_1,\ldots,m_s\rangle\).
Maximal abelianness gives
\[
 C_H(M)=M.
\]
Consequently, $M$ is defined in $H$ by the finite system
$[x,m_i]=1$ for $1\leq i\leq s$. If $d$ is the order of the
torsion subgroup of $M$, then Lemma~\ref{le:e-def-abelian}
shows that $M^d$ is nontrivial free abelian of finite rank and
e-definable in $M$, hence in $H$. Thus $H\in{\mathcal ED}$.
\end{proof}

Below \emph{hyperbolic group} means a finitely generated word-hyperbolic
group. 

\begin{proposition}\label{thm:hyperbolic}
Every infinite hyperbolic group $G$ contains an e-definable infinite
cyclic subgroup. If $G$ is torsion-free, then
for every $1\ne g\in G$ the centralizer $\Cent_G(g)$ is itself infinite
cyclic and is defined by the single equation $[x,g]=1$.
\end{proposition}

\begin{proof}
We use two standard facts: an infinite hyperbolic group contains an
element $g$ of infinite order, and $\langle g\rangle$ has finite index in
$\Cent_G(g)$; see \cite[Theorems~6.5.1 and~6.5.10]{Loh}.

Choose a maximal abelian subgroup $M$ of $G$ containing $g$. Such a
subgroup exists by Zorn's lemma, since a union of a chain of abelian
subgroups containing $g$ is again abelian. We have
\[
\langle g\rangle\leq M\leq\Cent_G(g).
\]
Hence $\langle g\rangle$ has finite index in $M$. Thus $M$ is
finitely generated abelian of rank one and is e-definable by
Proposition~\ref{thm:maximal}. If $d$ is the order of its torsion
subgroup, Lemma~\ref{le:e-def-abelian} gives an infinite cyclic
subgroup $M^d$ e-definable in $M$, hence in $G$.

If $G$ is torsion-free, every nonidentity element has infinite order.
Its centralizer is torsion-free and virtually infinite cyclic, hence
infinite cyclic, by the elementary classification of virtually cyclic
groups. Its defining equation is $[x,g]=1$.
\end{proof}

A finitely generated group is \emph{toral relatively hyperbolic}
if it is torsion-free and hyperbolic relative to a finite collection
of finitely generated abelian subgroups. We use the standard facts
that its maximal abelian subgroups are finitely generated and
malnormal; see \cite[Section~2, p.~3490]{GL2015}.
Here a subgroup \(M\leq G\) is \emph{malnormal} if
\(M\cap u^{-1}Mu=\{1\}\) for every \(u\notin M\).

\begin{proposition}
Let \(G\) be a nontrivial toral relatively hyperbolic group.
For every \(g\in G\setminus\{1\}\), the centralizer \(C_G(g)\)
is an infinite free abelian group of finite rank and is
\(e\)-definable in \(G\) by the single equation
\[
 [x,g]=1.
\]
\end{proposition}

\begin{proof}
Fix \(g\neq1\), and choose a maximal abelian subgroup \(M\)
containing \(g\). We claim that \(C_G(g)=M\).
Since \(M\) is abelian, \(M\subseteq C_G(g)\).
Conversely, if \(x\in C_G(g)\), then \(x^{-1}gx=g\), and hence
\[
 1\neq g\in M\cap x^{-1}Mx.
\]
Malnormality of \(M\) implies \(x\in M\), proving the claim.

The subgroup \(M\) is finitely generated torsion-free and abelian.
This proves the proposition.
\end{proof}

A \emph{limit group} here means a finitely generated fully residually
free group. Thus, for every finite set $S\subseteq L\setminus\{1\}$,
there is a homomorphism $\varphi:L\to F$ to a nonabelian free group such
that $\varphi(s)\ne1$ for all $s\in S$.

\begin{proposition}\label{thm:limit}
Let $L$ be a nontrivial limit group. For every $1\ne g\in L$,
\[
\Cent_L(g)\cong\Z^{r(g)},\qquad 1\leq r(g)<\infty.
\]
This centralizer is maximal abelian and e-definable by $[x,g]=1$.
In particular, every nontrivial limit group contains an e-definable
free abelian subgroup of positive finite rank.
\end{proposition}

\begin{proof}
Residual freeness implies that $L$ is torsion-free: a nonidentity torsion
element could not have nonidentity image in a free group.
We first show that $\Cent_L(g)$ is abelian. If $x,y\in\Cent_L(g)$ and
$[x,y]\ne1$, choose $\varphi:L\to F$ that is nontrivial on both $g$ and
$[x,y]$. The elements $\varphi(x)$ and $\varphi(y)$ lie in the cyclic
centralizer of $\varphi(g)\ne1$ in $F$. They therefore commute,
contradicting $\varphi([x,y])\ne1$.

Every abelian subgroup of a limit group is finitely generated.
This follows from the embedding in a finite sequence of extensions
of centralizers of a finitely generated free group established
in~\cite{KM1998}. Hence $\Cent_L(g)$ is finitely generated,
torsion-free and abelian. It contains the infinite cyclic subgroup
$\langle g\rangle$, so its rank is positive.
Any abelian subgroup containing $\Cent_L(g)$ still centralizes $g$,
which proves maximality. The equation $[x,g]=1$ defines this
centralizer in $L$.
\end{proof}

Another large subclass of ${\mathcal ED}$ is formed by right-angled Artin groups. 

Let $\Gamma=(V,E)$ be a finite simple graph with $V\ne\varnothing$. Its
right-angled Artin group is
\[
A_\Gamma=\langle V\mid [u,v]=1\text{ whenever }\{u,v\}\in E\rangle.
\]
For $S\subseteq V$, write $A_S=\langle S\rangle\leq A_\Gamma$; it is
naturally the right-angled Artin group on the induced subgraph with vertex
set $S$. Indeed, sending every vertex outside $S$ to $1$ defines a
retraction onto that group and proves that its natural map into
$A_\Gamma$ is injective.
A \emph{clique} is a set of pairwise adjacent vertices. It is
\emph{maximal} if no strictly larger clique contains it.

\begin{proposition}\label{thm:raag}
Let $\Sigma=\{v_1,\ldots,v_r\}$ be a maximal clique in $\Gamma$.
Then $r\geq1$ and
\begin{equation}\label{eq:clique-centralizer}
A_\Sigma\cong\Z^r,\qquad
A_\Sigma=\Cent_{A_\Gamma}(v_1,\ldots,v_r).
\end{equation}
In particular, $A_\Sigma$ is maximal abelian and is e-definable by
\[
[x,v_1]=\cdots=[x,v_r]=1.
\]
Thus every nontrivial right-angled Artin group on a finite graph contains
an e-definable free abelian subgroup of positive finite rank.
\end{proposition}
\begin{proof}
By \cite[Lemma~2.3, p.~156, and Lemma~2.1(7), p.~155]{DKR2008},
the centralizer of $A_\Sigma$ is generated by the vertices that
commute with every vertex of $\Sigma$. These are precisely the
vertices of $\Sigma$, since a vertex outside $\Sigma$ with this
property would enlarge the clique. Thus
$\Cent_{A_\Gamma}(A_\Sigma)=A_\Sigma$.
Since the induced graph on $\Sigma$ is complete,
$A_\Sigma\cong\Z^r$. The centralizer equality proves maximal
abelianness and gives the stated finite system of defining equations.
\end{proof}

Now we consider free products with amalgamation
\[
G=A*_C B.
\]
 We begin with two centralizer
statements.  
If $C$ is malnormal in one factor, then every hyperbolic element has an
infinite cyclic centralizer. If $C$ is abelian, then every element $ab$, with
$a\in A\setminus C$ and $b\in B\setminus C$, has an infinite abelian
centralizer; this centralizer is finitely generated when $C$ is finitely
generated.  These results are known, see below.

Recall that $C$ is \emph{malnormal in $A$} if
\[
C\cap uCu^{-1}=\{1\}\qquad\text{for every }u\in A\setminus C.
\]

\begin{proposition}[Malnormal amalgamation in one factor]\label{thm:malnormal}
Suppose that $C$ is proper in both $A$ and $B$, and that $C$ is malnormal
in at least one of the factors. Then every hyperbolic element $g\in G$
has an infinite cyclic centralizer. In particular, for every
$a\in A\setminus C$ and $b\in B\setminus C$,
\[
\Cent_G(ab)=\langle ab\rangle\cong\Z.
\]
Thus $G$ contains an infinite cyclic subgroup defined, with a parameter,
by the single equation $[x,ab]=1$.
\end{proposition}

\begin{proof}
We use the Bass--Serre tree $\mathcal T$ of the amalgam: its vertices
are $G/A\sqcup G/B$, and the edge $uC$ joins $uA$ to $uB$.
The action preserves the two vertex types and has no inversions.
For the tree construction and the axis theorem, see
\cite[Chapter~I, 4.1, Theorem~7, and 6.4, Proposition~24]{Serre}.
By symmetry, assume that $C$ is malnormal in $A$.

Let $L$ be the axis of a hyperbolic element $g$.
Every element of $\Cent_G(g)$ preserves $L$ and acts on it by a
translation: a reflection would conjugate the nonzero translation
induced by $g$ to its inverse. Signed translation distance therefore
defines a homomorphism
\[
\tau:\Cent_G(g)\longrightarrow\Z.
\]
An element of its kernel fixes $L$ pointwise. At any vertex of $L$
of type $A$, it fixes the two distinct incident edges of $L$.
Their stabilizers have trivial intersection, since, after conjugation,
this intersection is $C\cap uCu^{-1}$ for some $u\in A\setminus C$.
Thus $\tau$ is injective. Its image contains $\tau(g)\ne0$, so
$\Cent_G(g)$ is infinite cyclic.

Now put $t=ab$. The length-two path from $B$ to $tB=aB$ passes through
$A$ and contains the standard edge $C$. Its translates by the powers
of $t$ concatenate without backtracking, since $a\notin C$ and
$b\notin C$. They form the axis of $t$, on which $t$ translates by
two edges. Since the action preserves vertex types, every translation
induced by $\Cent_G(t)$ has even length. Orienting the axis so that
$\tau(t)=2$, we obtain $\tau(\Cent_G(t))=2\Z$.
Injectivity of $\tau$ now gives $\Cent_G(t)=\langle t\rangle$.
This subgroup is defined by $[x,t]=1$.

The related result \cite[Proposition~1]{GKM1995} states that malnormal
subgroups of the other factor remain malnormal in the amalgam;
the centralizer assertion above follows from the axis argument.
\end{proof}

\begin{proposition}[Abelian amalgamation]\label{thm:abelian}
Suppose that $C$ is abelian and proper in both $A$ and $B$. For every
$a\in A\setminus C$ and $b\in B\setminus C$, the subgroup
\[
\Cent_G(ab)=\bigl(C\cap\Cent_G(ab)\bigr)\times\langle ab\rangle
\]
is infinite abelian and is defined by the equation $[x,ab]=1$.
If $C$ is finitely generated, this centralizer is finitely generated
abelian of positive rank. If $C\cong\Z^r$, where $r\geq0$, then
\[
\Cent_G(ab)\cong\Z^s
\qquad\text{for some }1\leq s\leq r+1.
\]
\end{proposition}

\begin{proof}
Put $t=ab$ and consider the Bass--Serre tree described in the preceding
proof. The translates of the path $B,A,aB=tB$ again concatenate
without backtracking. Hence $t$ has translation length two and its
axis $L$ contains the standard edge $C$, whose stabilizer is $C$.
These facts use only $a\notin C$ and $b\notin C$, not malnormality;
see \cite[Chapter~I, 4.1, Theorem~7, and 6.4, Proposition~24]{Serre}.

Every $z\in\Cent_G(t)$ acts on $L$ by a translation. Its signed
translation distance is even, because the action preserves vertex
types. Consequently, for some $n\in\Z$, the element $zt^{-n}$ fixes
$L$ pointwise. It fixes the standard edge and still centralizes $t$,
so
\[
zt^{-n}\in K:=C\cap\Cent_G(t).
\]
It follows that $\Cent_G(t)=K\langle t\rangle$. The two factors
commute by the definition of $K$, and their intersection is trivial:
every nonidentity power of $t$ translates $L$, whereas every element
of $K$ fixes the standard edge. Thus
\[
\Cent_G(t)=K\times\langle t\rangle.
\]

Since $K\leq C$ and $C$ is abelian, this centralizer is abelian.
It is infinite because $t$ has infinite order, and it is defined by
$[x,t]=1$. If $C$ is finitely generated, then so is its subgroup $K$,
giving a finitely generated abelian centralizer of positive rank.
If $C\cong\Z^r$, then $K\cong\Z^q$ for some $0\leq q\leq r$,
and the displayed direct product has rank $q+1$.
For the related algebraic description of commuting pairs, see
\cite[Theorem~4.5]{MKS}.
\end{proof}

\section{The Diophantine problem in various wreath products}

In this section we combine Theorems \ref{thm:main},  \ref{cor:wreath-diophantine-reduction}, and \ref{main} to prove undecidability of the Diophantine problem in various wreath products of groups. 

\begin{theorem} \label{th:general}
Let $G=B\wr H$ be a restricted wreath product of finitely generated
groups $B$ and $H$. Then the Diophantine problem in $G$ is undecidable
if at least one of the following conditions holds.
\begin{enumerate}[label=\textup{(\roman*)}]
\item either $\Dio(H)$ or $\Dio(B)$ is undecidable;
\item $B, H \in {\mathcal I}$;
\item $B \in \mathcal I$, $H \in {\mathcal ED}$.
\end{enumerate}
\end{theorem}

\begin{proof}
Cases \textup{(i)} and \textup{(ii)} are
Theorems~\ref{thm:main} and~\ref{cor:wreath-diophantine-reduction},
respectively. In case \textup{(iii)}, the indicability of $B$ gives an
infinite cyclic retract $C\leq B$ by
Proposition~\ref{prop:indicable-characterization}.
Theorem~\ref{retract}, applied to the retraction $B\to C$ and the
identity of $H$, makes $C\wr H$ a retract of $B\wr H$.
Since $H\in{\mathcal ED}$, Theorem~\ref{main} implies that
$\Dio(C\wr H)$ is undecidable. The retract principle
(Proposition~\ref{prop:retract-principle}) gives undecidability in
$B\wr H$.
\end{proof}

By item \textup{(i)}, if either of the groups $B$ and $H$ has an  undecidable Diophantine problem, then $\Dio(B\wr H)$ is also undecidable.
We therefore consider infinite finitely generated groups $B$ and $H$
with decidable Diophantine problems. The following families provide
examples (in each item, we restrict to infinite finitely generated
groups):

\begin{enumerate}[label=\textup{(\arabic*)}]
\item Virtually abelian groups (Ershov~\cite{Ershov1972});
\item Free groups (Makanin~\cite{Makanin1982});
\item Hyperbolic groups (Rips and Sela~\cite{RipsSela1995} in the
torsion-free case, and Dahmani and Guirardel~\cite{DG2010} with torsion);
\item Toral relatively hyperbolic groups (Dahmani~\cite{Dahmani2009});
\item Right-angled Artin groups on finite graphs
(Diekert and Muscholl~\cite{DiekertMuscholl2006}).
\end{enumerate}

By Section~\ref{se:ED}, every infinite group in these families
belongs to ${\mathcal ED}$. Hence, for every such $H$ and every finitely
generated indicable group $B$, the Diophantine problem in $B\wr H$
is undecidable by Theorem~\ref{th:general}\textup{(iii)}.

\end{document}